\documentclass{article}
\usepackage{amsmath,amsthm,amssymb,latexsym,url}
\theoremstyle{plain}

\newtheorem{prop}{Proposition}

\title{{\bf A simple counterexample to\\Havil's ``reformulation''\\of the Riemann Hypothesis}}
\author{Jonathan Sondow}
\date{}

\begin{document}
\maketitle
\begin{abstract}
The Riemann Hypothesis (RH) is the most famous unsolved problem in mathematics. It is an assertion about the Riemann zeta function which, if true, would provide an optimal form of the Prime Number Theorem. In his book {\it Gamma: Exploring Euler's Constant}, J.~Havil makes a conjecture which he claims is ``a tantalizingly simple reformulation'' of the RH. We first explain the RH, Havil's conjecture, and the connection between them. Then we give a simple example which we show is a counterexample to Havil's conjecture, but not to the~RH. Finally, we prove that a weakened form of Havil's conjecture is a true reformulation of the RH.
\end{abstract}

\vspace{1em}
{
\setlength{\parindent}{64pt}
\small
Jonathan Sondow\par 209 West 97th Street\par New York, NY 10025, USA\par email: {\tt jsondow@alumni.princeton.edu} }

\section{Introduction} \label{SEC: introd}

The Riemann Hypothesis (RH) is the most famous unsolved problem in mathematics. It is an assertion about the Riemann zeta function which, if true, would provide an optimal form of the Prime Number Theorem.

In his book {\it Gamma: Exploring Euler's Constant}, J.~Havil claims that the following conjecture is ``a tantalizingly simple reformulation'' of the RH.\\

\noindent\textbf{Havil's Conjecture.} {\it If
\begin{equation*} 
	\sum_{n=1}^\infty \frac{(-1)^n}{n^a} \cos(b\ln n)=0 \quad { and}  \quad \sum_{n=1}^\infty \frac{(-1)^n}{n^a} \sin(b\ln n)=0 
\end{equation*}
for some pair of real numbers $a$ and $b$, then $a=1/2$.}\\

We first explain the RH and its connection with Havil's conjecture. Then we show that the pair of real numbers $a=1$ and \mbox{$b=2\pi/\!\ln2$} is a counterexample to Havil's Conjecture, but not to the~RH. Finally, we prove that Havil's Conjecture becomes a true reformulation of the RH if his conclusion ``\emph{then} $a=1/2$'' is weakened to ``\emph{then $a=1/2$ or} $a=1$.''

\section{The Riemann Hypothesis} \label{SEC: rh}

In $1859$ Riemann published a short paper \emph{On the number of primes less than a given quantity}~\cite{riemann}, his only one on number theory. Writing $s$ for a complex variable, at first he assumes that its real part $\Re(s)$ is greater than~$1$, and he begins with \emph{Euler's product-sum formula}
\begin{equation*}
	 \prod_p \frac{1}{1-\dfrac{1}{p^s}}=\sum_{n=1}^\infty \frac{1}{n^s}  \qquad (\Re(s)>1). 
\end{equation*}
Here the product is over all primes $p$. To see the equality, expand each factor $1/(1-1/p^s)$ in a geometric series, multiply them together, and use unique prime factorization.

Euler proved his formula only for real $s\ (>1)$. He used it to give a new proof of Euclid's theorem on the infinitude of the primes: if there were only finitely many, then as $s\to1^+$ the left-hand side of the formula would approach a finite number while the right-hand side approaches the harmonic series $\sum1/n = \infty$. Going further than Euclid, Euler also used his formula to show that, unlike the squares, the primes are so close together that
$$ \frac12+\frac13+\frac15+\frac17+\frac{1}{11}+\frac{1}{13}+\frac{1}{17}+\frac{1}{19}+\frac{1}{23}+\frac{1}{29}+\frac{1}{31}+\frac{1}{37}+\cdots = \infty.$$

Aiming to go even further, Riemann develops the properties of the \emph{Riemann zeta function} $\zeta(s)$. He defines it first as the sum of the series in Euler's formula,
\begin{equation}
	\zeta(s)=\sum_{n=1}^\infty \frac{1}{n^s} = \frac{1}{1^s}+\frac{1}{2^s}+\frac{1}{3^s}+\cdots \qquad (\Re(s)>1), \label{EQ: zeta}
\end{equation}
which converges by comparison with the $p$-series (here $p=\Re(s)$)
$$\sum\left|\frac{1}{n^s}\right|=\sum\frac{1}{n^{\Re(s)}}.$$

Using other formulas for $\zeta(s)$, Riemann extends its definition to the whole complex plane, where it has one singularity, a~simple pole at $s=1$, reflecting Euler's observation that $\zeta(s)\to\sum1/n = \infty$ as $s\to1^+$.

Riemann analyzes the \emph{zeros of} $\zeta(s)$, which he calls the roots of the equation \mbox{$\zeta(s)=0$}. He shows that there are none in the right half-plane \mbox{$\{s:\Re(s)>1\}$}, and that the only ones in the left half-plane \mbox{$\{s:\Re(s)<0\}$} are the negative even integers \mbox{$s=-2,-4,-6,\ldots$}. (These real zeros had been found by Euler more than a century earlier---see \cite{ayoub}.)

Turning his attention to the zeros in the closed strip \mbox{$\{s:0\le\Re(s)\le1\}$}, Riemann proves that they are symmetrically located about the vertical line $\Re(s)=1/2$. Using an integral, he estimates the number of them with imaginary part between $0$ and some bound $T$. Then he says:

\def\changemargin#1#2{\list{}{\rightmargin#2\leftmargin#1}\item[]}
\let\endchangemargin=\endlist 
\begin{changemargin}{0in}{0in}
\begin{quotation}
\noindent One finds in fact about this many [on the line] within these bounds and it is very likely that all of the [zeros in the strip are on the line]. One would of course like to have a rigorous proof of this; however, I have tentatively put aside the search for such a proof after some fleeting vain attempts $\ldots$.
\end{quotation}
\end{changemargin}
Thus was born the now famous and still unproven RH.\\

\noindent\textbf{The Riemann Hypothesis.} {\it If $\zeta(s)=0$ and $s\neq-2,-4,-6,\ldots$, then \mbox{$\Re(s)=1/2$}.}\\

Around 1896 Hadamard and de la Vall\'ee Poussin, independently, took a step in the direction of the RH by proving that \mbox{$\zeta(s)\neq0$} on the line \mbox{$\Re(s)=1$}. This was a crucial ingredient in their proofs of the \emph{Prime Number Theorem} (PNT), which estimates $\pi(x)$, \emph{the number of primes $p$ with} $2\le p\le x$. Conjectured by Gauss in 1792 at the age of 15, the PNT says that
$$\pi(x) \sim \frac{x}{\ln x}.$$
That is, as $x\to\infty$ the limit of the quotient
$$\frac{\pi(x)}{x/\!\ln x}$$
exists and equals~$1$. More accurately, Gauss guessed that $\pi(x)\sim \text{Li}(x)$, where
$$\text{Li}(x) = \int_2^x \frac{dt}{t}$$
is the \emph{logarithmic integral}, which is equal to $x/\!\ln x$ plus a smaller quantity. If the RH is true, then the PNT's estimate $\pi(x) \sim \text{Li}(x)$ comes with a good bound on the error $\pi(x) - \text{Li}(x)$. Namely, the RH implies that the inequality
$$\left| \pi(x) - \text{Li}(x) \right| < Cx^{\frac12+\epsilon}$$
holds for any $\epsilon > 0$ and all $x\ge2$, where $C$ is a positive number that depends on $\epsilon$ but not on $x$. In fact, this inequality is \emph{equivalent} to the RH, and the appearance of $1/2$ in both is not a coincidence.

Since \mbox{$\zeta(s)\neq0$} whenever \mbox{$\Re(s)=1$}, the symmetry of the zeros in the strip \mbox{$\{s:0\le\Re(s)\le1\}$} implies that they lie in the \emph{open} strip $\{s:0<\Re(s)<1\}$. In 1914 Hardy proved that infinitely many of them lie on the line $\Re(s)=1/2$. That lends credence to the RH, but of course does not prove it.

Hardy's result was later improved by Selberg and others, who showed that a positive percentage of the zeros in the strip lie on its center line. Using computers and further theory, the first $10^{13}$ zeros have been shown to lie on the line. For more on the RH, the PNT, and their historical background, see \cite{bcrw}, \cite{sautoy}, and \cite[Chapter~16]{havil}.

In order to relate the RH to Havil's Conjecture, we need to introduce a function closely related to $\zeta(s)$.

\section{The Alternating Zeta Function} \label{SEC: alt zeta}

\noindent For $s$ with positive real part, the \emph{alternating zeta function} $\zeta_*(s)$ (also known as the \emph{Dirichlet eta function} $\eta(s)$) is defined as the sum of the alternating series
\begin{equation*} 
	\zeta_*(s)=\sum_{n=1}^\infty \frac{(-1)^{n-1}}{n^s} = \frac{1}{1^s}-\frac{1}{2^s}+\frac{1}{3^s}-\frac{1}{4^s}+\frac{1}{5^s}-\frac{1}{6^s}+\cdots \quad (\Re(s)>0), 
\end{equation*}
which we now show converges.

Since $\Re(s)>0$, the $n$th term approaches $0$ as $n\to\infty$, so that we only need to show convergence for the series formed by grouping the terms in odd-even pairs. Writing each pair as an integral
$$\frac{1}{n^s} - \frac{1}{(n+1)^s} =s \int_n^{n+1} \frac{dx}{x^{s+1}} $$
with $n$ odd, we set $\sigma=\Re(s)$ and bound the integral:

$$\left| \int_n^{n+1} \frac{dx}{x^{s+1}}\right| \le \int_n^{n+1} \frac{dx}{\left|x^{s+1}\right|} = \int_n^{n+1}\frac{dx}{x^{\sigma+1}} < \frac{1}{n^{\sigma+1}}.$$
As the series $\sum n^{-\sigma-1}$ converges, so does the alternating series for $\zeta_*(s)$.

When $\Re(s)>1$ the alternating series converges absolutely, and so we may rewrite it as the difference
\begin{eqnarray*}
\zeta_*(s)     &=& \frac{1}{1^s}+\frac{1}{2^s}+\frac{1}{3^s}+\frac{1}{4^s}+\frac{1}{5^s}+\frac{1}{6^s}+\cdots - 2\left(\frac{1}{2^s}+\frac{1}{4^s}+\frac{1}{6^s}+\cdots\right) \\
                      &=& \frac{1}{1^s}+\frac{1}{2^s}+\frac{1}{3^s}+\frac{1}{4^s}+\frac{1}{5^s}+\frac{1}{6^s}+\cdots - \frac{2}{2^s} \left(\frac{1}{1^s}+\frac{1}{2^s}+\frac{1}{3^s}+\cdots\right)\\
                        &=& \zeta(s) - \frac{2}{2^s}\zeta(s).
\end{eqnarray*}
Thus the alternating zeta function is related to the Riemann zeta function by the simple formula
\begin{equation}
	\zeta_*(s) = \left(1-2^{1-s}\right) \zeta(s). \label{EQ: 2zetas}
\end{equation}
We derived it for $s$ with $\Re(s)>1$, but a theorem in complex analysis guarantees that the formula then remains valid over the whole complex plane.

At the point $s=1$, the simple pole of $\zeta(s)$ is cancelled by a zero of the factor $1-2^{1-s}$. This agrees with the fact that $\zeta_*(s)$ is finite at $s=1$. Indeed, $\zeta_*(1)$ is equal to Mercator's alternating harmonic series$$\zeta_*(1)=1-\frac12+\frac13-\frac14+\frac15-\frac16+\cdots=\ln2.$$

The product formula \eqref{EQ: 2zetas} shows that $\zeta_*(s)$ vanishes at each zero of the factor \mbox{$1-2^{1-s}$} with the exception of $s=1$. (This can also be proved without using~\eqref{EQ: 2zetas}---an elementary proof is given in~\cite{sondow}.) It is a nice exercise to show that the zeros of $1-2^{1-s}$ lie on the line $\Re(s)=1$, and occur at the points $s_k$ given by
$$s_k=1+i\frac{2\pi k}{\ln2} \qquad (k=0,\pm1,\pm2,\pm3,\ldots).$$
Thus $s_k$ is also a zero of $\zeta_*(s)$ if $k\neq0$.

Since $1-2^{1-s}\neq0$ when $\Re(s)\neq1$, relation~\eqref{EQ: 2zetas} also shows that $\zeta_*(s)$ and $\zeta(s)$ have the same zeros in the strip $\{s:0<\Re(s)<1\}$. The first one is
$$\rho_1 = 0.5+14.1347251417346937904572519835624702707842571156992\ldots i,$$
the Greek letter $\rho$ (rho) standing for root. Using a calculator, the reader can see it is likely that $\zeta_*(\rho_1)=0$. But be patient: at $s=\rho_1$ the alternating series for $\zeta_*(s)$ converges very slowly, because its $n$th term has modulus $\left|(-1)^{n-1}n^{-\rho_1}\right|$ $=n^{-1/2}$. For example, to get $n^{-1/2}<0.1$, you need $n>100$.

If we substitute the series for $\zeta_*(s)$ into equation~\eqref{EQ: 2zetas} and solve for $\zeta(s)$, we obtain the formula
\begin{equation}
	\zeta(s) = \frac{1}{1-2^{1-s}}\sum_{n=1}^\infty \frac{(-1)^{n-1}}{n^s} \qquad (\Re(s)>0,\ s\neq1). \label{EQ: explicit}
\end{equation}
Since the series converges whenever $\Re(s)>0$, the right-hand side makes sense for all $s\neq1$ with positive real part, the first factor's poles at $s=s_k\neq1$ being cancelled by zeros of the second factor. Thus the formula extends the definition \eqref{EQ: zeta} of $\zeta(s)$ to a larger domain.

We can now explain the relation between the RH and Havil's Conjecture.

\section{A Counterexample to Havil's Conjecture} \label{SEC: counterex}

\noindent Let's write $s=a+ib$, where $a$~and~$b$ are real numbers. Euler's famous formula
$$e^{ix} = \cos x + i\sin x$$
shows that
\begin{equation*}
\frac{1}{n^s} = \frac{1}{n^{a+ib}} = \frac{1}{n^a} e^{-ib\ln n} = \frac{1}{n^a} \left(\cos(b\ln n)-i\sin(b\ln n)\right). 
\end{equation*}
Now if $a=\Re(s)>0$, then
\begin{eqnarray}
	\zeta_*(s)=\sum_{n=1}^\infty \frac{(-1)^{n-1}}{n^s} \!\!&=&\!\!\sum_{n=1}^\infty \frac{(-1)^{n-1}}{n^a} \left(\cos(b\ln n)-i\sin(b\ln n)\right) \label{EQ: 2series}\\ 
										     \!\!&=&\!\!-\sum_{n=1}^\infty \frac{(-1)^n}{n^a} \cos(b\ln n) + i\sum_{n=1}^\infty \frac{(-1)^n}{n^a}\sin(b\ln n).\notag
\end{eqnarray}

Up to sign, the last two series are the real and imaginary parts of $\zeta_*(s)$. Hence $\zeta_*(s)=0$ if and only if both series vanish. Since they are the same series as in Havil's Conjecture and $\zeta_*(s)=0$ at $s=s_1=1+2\pi i/\!\ln2$, \emph{the pair of real numbers $a=1$ and $b=2\pi/\!\ln2$ is a counterexample to Havil's Conjecture.}

On the other hand, since the theorem of Hadamard and de la Vall\'ee Poussin says that $\zeta(s)$ has no zeros with real part equal to~$1$, \emph{the point \mbox{$s_1=1+2\pi i/\!\ln2$} is not a counterexample to the RH.} Therefore, Havil's Conjecture is \emph{not} a reformulation of the RH.

From looking at the two series in his conjecture it is not at all clear that they are equal to $0$ when $a=1$ and $b=2\pi/\!\ln2$. This shows the power of the alternate formulation $\zeta_*(s_1)=0$. 

To conclude, we give a \emph{true} reformulation of the RH.

\section{The RH Without Tears} \label{SEC: tears}

\noindent  Here is a corrected version of Havil's Conjecture.\\

\noindent\textbf{New Conjecture.} {\it If
\begin{equation} \label{EQ: reform2}
	\sum_{n=1}^\infty \frac{(-1)^n}{n^a} \cos(b\ln n)=0 \quad and  \quad \sum_{n=1}^\infty \frac{(-1)^n}{n^a} \sin(b\ln n)=0 
\end{equation}
for some pair of real numbers $a$ and $b$, then $a=1/2$ or $a=1$}.\\

Let's show that this is indeed a reformulation of the RH.

\begin{prop} \label{PROP: new conj}
The New Conjecture is true if and only if the RH is true.
\end{prop}
\begin{proof}
Suppose that the New Conjecture is true. Assume that $\zeta(s)=0$ and that \mbox{$s\neq-2,-4,-6,\ldots$}. By Riemann's results and the Hadamard-de la Vall\'ee Poussin theorem, $s$~lies in the open strip $\{s:0<\Re(s)<1\}$. Then relation~\eqref{EQ: 2zetas} gives $\zeta_*(s)=0$. Writing $s=a+ib$, equation~\eqref{EQ: 2series} yields the equalities in~\eqref{EQ: reform2}, and so by the New Conjecture $a=1/2$ or $a=1$. But $a=\Re(s)$ and $\Re(s)<1$. Hence $\Re(s)=1/2$. Thus the New Conjecture implies the RH. 

Conversely, suppose the RH is true. Assume $a$ and $b$ satisfy condition~\eqref{EQ: reform2}. In particular, both series in \eqref{EQ: reform2} converge, and so their $n$th terms tend to $0$ as $n\to\infty$. It follows that the sum of the squares of the $n$th terms, namely, $n^{-2a}$, also tends to $0$. Hence $a>0$. Then with $s=a+ib$ equation~\eqref{EQ: 2series} applies, and \eqref{EQ: reform2}~yields $\zeta_*(s)=0$. Now relation~\eqref{EQ: 2zetas} shows that $s$~is a zero of~$\zeta(s)$ or of~$1-2^{1-s}$. In the first case the RH says $a=1/2$, and in the second case we know $a=1$. Thus the RH implies the New Conjecture.
\end{proof}

\paragraph{Acknowledgments.} I am grateful to Jacques G\'elinas and Steven J. Miller for valuable suggestions and corrections, and to Jussi Pahikkala for helpful comments that led to improvements in the presentation.


\end{document}